\documentclass[12pt,a4paper]{article}

\usepackage[english]{babel}
\usepackage[T1]{fontenc}
\usepackage[utf8]{inputenc}

\usepackage{bm}
\usepackage{lmodern}
\usepackage[normalem]{ulem}

\usepackage[top=2.25cm, bottom=2.75cm, left=2.75cm, right=2.75cm]{geometry}
\usepackage{latexsym}
\usepackage{pdflscape}
\usepackage{fancyhdr}
\usepackage{setspace}

\usepackage{graphicx}
\usepackage[dvipsnames]{xcolor}
\usepackage{epstopdf}
\usepackage{epsf}
\usepackage{eepic}
\usepackage{tikz}
\usetikzlibrary{shadings,shapes,arrows,calc,positioning,shadings,decorations.markings, patterns}

\usepackage[numbers]{natbib}

\usepackage{amsmath,amssymb,amsfonts,amscd}
\usepackage{enumitem}
\usepackage{caption}
\usepackage[format=hang,margin=10pt]{subcaption}
\usepackage{listings}
\usepackage{varwidth}

\usepackage{algorithm}
\usepackage{algpseudocode}
\usepackage[numbered,framed]{matlab-prettifier}
\usepackage{amsthm}
\usepackage[linguistics]{forest}
\forestset{
nice empty nodes/.style={
    for tree={l sep*=2},
    delay={where content={}{shape=coordinate}{}}
},
}

\usepackage{changepage}
\usepackage{xifthen}
\usepackage{todonotes}

\usepackage{hyperref}
\usepackage{cleveref}

\numberwithin{equation}{section}
\setlist[itemize,1]{label=$\bullet$}
\setlist[itemize,2]{label=$\triangleleft$}
\setlist[enumerate,1]{label=(\roman*)}
\setlist[enumerate,2]{label=(\arabic*)}

\definecolor{TUIl-orange}{RGB}{255, 121, 0}
\definecolor{TUIl-titleblue}{RGB}{0, 68, 121}
\definecolor{TUIl-textblue}{RGB}{0, 51, 88}
\definecolor{TUIl-green}{RGB}{0, 116, 122}
\definecolor{TUIl-grey}{RGB}{165, 165, 165}

\hypersetup{
    breaklinks=true,
    unicode=true,
    pdfpagelayout=OneColumn,
    bookmarksnumbered=true,
    bookmarksopen=true,
    bookmarksopenlevel=0,
    pdfborder={0 0 0},  
    colorlinks=true,
    linkcolor=Cerulean,    
    citecolor=Cerulean,
}

\algnewcommand\algorithmicinput{\textbf{Input:}}
\algnewcommand\AlgInput{\item[\algorithmicinput]}
\algnewcommand\algorithmicoutput{\textbf{Output:}}
\algnewcommand\AlgOutput{\item[\algorithmicoutput]}

\lstMakeShortInline"
\newtheoremstyle{dotless}{}{}{\itshape}{}{\bfseries}{}{ }{}
\newtheoremstyle{no-italic}{}{}{}{}{\bfseries}{}{ }{}
\theoremstyle{dotless}
\newtheorem{Theorem}{Theorem}[section]
\newtheorem{Example}[Theorem]{Example}
\newtheorem{Lemma}[Theorem]{Lemma}
\newtheorem{Definition}[Theorem]{Definition}
\newtheorem{Assumption}{Assumption}
\newtheorem{Remark}[Theorem]{Remark}
\newtheorem{Proposition}[Theorem]{Proposition}

\newtheorem{Test Instance}[Theorem]{Test Instance}
\theoremstyle{no-italic}

\newcommand*{\R}{\mathbb{R}}

\newcommand*{\N}{\mathbb{N}}

\DeclareMathOperator{\cl}{cl}			% closure
\usepackage{tikz}

\def\R{{\mathbb R}}

\def\dom{\textup{dom }}

\def\epi{\textup{epi }}
\def\Int{\textup{int }}
\def\bd{\textup{bd }}
\def\cl{\textup{cl }}
\def\conv{\textup{conv }}
\def\gph{\textup{gph }}
\def\epi{\textup{epi }}
\title{On Clarke's Subdifferential of Marginal Functions}
\author{Gemayqzel Bouza  \thanks{Faculty of Mathematics and Computer Science, University of Havana, Havana, Cuba,
{\texttt{gema@matcom.uh.cu}}} \and Ernest Quintana \thanks{Institute for Mathematics, Technische Universität Ilmenau, Ilmenau, Germany,
{\texttt{ernest.quintana-aparicio@tu-ilmenau.de}}} \and Christiane Tammer \thanks{Institute of Mathematics, Martin-Luther-Universität Halle-Wittenberg, Halle, Germany, {\texttt{christiane.tammer@mathematik.uni-halle.de}}}}
\date{}

\begin{document}

%%%%%%%%%%%%%%%%
%% Title page %%
%%%%%%%%%%%%%%%%
\maketitle

%%%%%%%%%%%%%%
%% Abstract %%
%%%%%%%%%%%%%%
\begin{abstract}
In this short note, we derive an upper estimate of Clarke's subdifferential of marginal functions in Banach spaces. The structure of the upper estimate is very similar to other results already obtained in the literature. The novelty lies on the fact that we derive our assertions in general Banach spaces, and avoid the use of the Asplund assumption.
\end{abstract}

%%%%%%%%%%%%%%%%%%%%%%%%%%%%%%%%%%
%% Key words and classification %%
%%%%%%%%%%%%%%%%%%%%%%%%%%%%%%%%%%
Generalized differentiation.

\vspace{2ex} \noindent {\small\textbf{Mathematics subject classifications (MSC 2010):} 90C30, 49J52, 49J53.}

%%%%%%%%%%%%%
%% Content %%
%%%%%%%%%%%%%

\section{Introduction}

Marginal functions are maps $\varphi : X \rightarrow \overline{\R}$ of the form 

\begin{equation}\label{eq: marginal function}
\varphi(x):= \inf_{y\in F(x)} f(x,y),
\end{equation} where $X$ and $Y$ are normed spaces, $\overline{\R}: = [-\infty, + \infty]$ is the extended set of real numbers, $F: X\rightrightarrows Y$ is a given set-valued mapping, and  $f: X \times Y \to \overline{\R}$ is a given functional. Due to its numerous applications, functionals of this type are among the most important ones in variational analysis, mathematical programming, and control theory. The study of subdifferentiability properties of  $\varphi,$ and specifically the computation of its subdifferential (in a specific sense), is therefore a relevant research topic. 

Under convexity assumptions on the graph of $F$ and the functional $f,$ the Fenchel subdifferential of $\varphi$ is analyzed in \cite{AY2015,BondarevskyLeschovMinchenko2016,Luderer1991,Minchenko1991}. In this case, because of the nice convexity structure of the problem, the derived formulas are exact under mild conditions. On the other hand, the nonconvex setting has also been studied extensively, see for example \cite{HUrruty1978,Mordukhovich1,MNY2009,Thibault1991}. However, in this case, exactness is already too difficult to ensure.  These results are then focused on computing upper estimates of different types of subdifferentials of the marginal function. In \cite{bot2012}, Bo{\c{t}} studied  Clarke's subdifferential of $\varphi$ when both $X$ and $Y$ are finite dimensional, $F$ is constant, and $f$ is given as the pointwise maxima of a finite number of continuously differentiable functions.  A similar study was considered by Hiriart-Urruty in \cite{HUrruty1978} for the case in which  $\varphi$ is locally Lipschitz (in particular, when $F$ is locally Lipschitz). Later, in \cite{Mordukhovich1,MNY2009,Thibault1991}, upper estimates of Fr\'{e}chet and Mordukhovich's subdifferential of $\varphi$ were derived assuming that $X$ and $Y$ are Asplund spaces and that $F$ has a closed graph. 

To the best of our knowledge, upper estimates for Clarke's subdifferential of $\varphi$ have not been studied in the literature without the Asplund assumption on the underlying spaces. In this paper, we address this case. Our results rely only on the Banach structure of $X$ and $Y,$ and on the Aubin property of $F$ at notable points. 

The rest of the paper is organized as follows: In Section 2, we establish important notations and definitions used throughout the text. In Section 3, the main result is derived. We conclude with some remarks and future lines of research in Section 4.

\section{Preliminaries}

We start by establishing the main notations used in the paper. Given a normed space $(X, \|\cdot\|_X),$ we will denote by $(X^*,\|\cdot\|_X^*)$ its topological dual. In addition, the closed unit balls in $X$ and $X^*$ will be denoted as $\mathbb{B}_X$ and $\mathbb{B}^*_X$ respectively. We omit the subscript $X$  if there is no risk of confusion. For a nonempty set $A\subseteq X$, $\Int A$, $\cl A$, $\bd A$, $\conv A $ stand for the interior, closure, boundary and convex hull of $A$, respectively. Furthermore, if $B \subseteq X^*,$ we denote the closure of the convex hull of $B$ in the weak$^*$ topology of $X^*$ by $\overline{\operatorname{conv}}^* B$.  Throughout the rest of this section, we assume that normed spaces $X$ and $Y$ are given.

\begin{Definition}\label{def Lipschitz-like} 

Let $F: X\rightrightarrows Y$ be a set-valued mapping. 

\begin{enumerate}
    \item  The graph of $F$ is the set defined  as 

\[
\gph F:=\big\{(x,y)\in X\times Y
\mid y\in F(x)\big\}.
\]

\item The domain of $F$ is the set given by 

$$\dom F:= \{x \in X \mid F(x) \neq \emptyset\}.$$

\item We say that $F$ satisfies the Aubin property at $(\bar{x},\bar{y}) \in \gph F$ if there exists $\ell\geq 0,$ together with neighborhoods $U$ of $\bar{x}$ and $V$ of $\bar{y},$ such that

\begin{equation*}
\forall \; x,x^{\prime}\in U: \qquad F(x)\cap V\subseteq F(x^{\prime})+\ell\left\Vert x-x^{\prime}\right\Vert _{X} \mathbb{B}_Y 
\end{equation*}
\end{enumerate}

\end{Definition}

\begin{Definition}
Let $f:X\rightarrow \overline{\R}$ be an extended real valued functional.

\begin{enumerate}
\item The domain of $f$  is the set 

\[
\operatorname{dom}f:=\{x\in X\mid f(x)<+\infty\}.
\]
 \item The epigraph of $f$ is the set defined as 
 
 \[
\epi f:=\{(x,t)\in 
X\times \R \mid f(x)\leq t\}.
\]

\item We say that $f$ is convex if $\epi f$ is a convex set.

\item We say that $f$ is Lipschitz on a set $A \subseteq X$ provided that $f$ is finite on $A$ and there exists $L>0$ 	such that
\[
\forall \; x,x' \in A:	|f(x)-f(x')|\leq L \| x-x'\|_X.
\]  We say that $f$ is locally Lipschitz at $\bar{x}$ if there is a neighborhood $U$ of $\bar{x}$ such that $f$ is Lipschitz on $U$. In addition, $f$ is said to be locally Lipschitz on $A$, if $f$ is locally Lipschitz at every point $x\in A$. Thus, in particular, $A\subseteq\operatorname*{int}\operatorname*{dom}f$.
\end{enumerate}
\end{Definition}

An important class of extended real valued functionals that will be useful in the sequel are the so called support functions of subsets in a dual space. Formally, for a set $A\subseteq X^*,$ this is the functional  $\sigma_A: X \rightarrow \overline{\R}$ defined as

\begin{equation*}
\sigma_A(x):=\sup\limits_{x^*\in A}\left\langle x^{\ast},x \right\rangle.
\end{equation*}

Next, we present the main concepts of generalized differentiation that will be employed in the paper. The material in this part is mostly classical, and we refer the reader to \cite{Clarke1} for a more in depth discussion of these concepts and their properties.
\begin{Definition} Let  $f: X\rightarrow \overline{\R}$ be locally Lipschitz at a given point $\bar{x}\in \dom f$. 
\begin{enumerate}
\item For each $u\in X$, the generalized directional derivative of $f$ at $\bar{x}$ in the direction $u$ is defined by
\begin{equation*}
f^{\circ}(\bar{x},u):=\limsup\limits_{x \rightarrow \bar{x},t\downarrow 0}\frac{f(x+tu)-f(x)}{t}.
\label{gendirder}
\end{equation*}
\item The Clarke subdifferential of $f$ at $\bar{x}\in \dom f$ is the set defined as
\begin{equation*}
\partial^{\circ} f(\bar{x}):=\{x^*\in X^*\mid \forall \;u\in X: f^{\circ}(\bar{x},u)\geq \langle x^*,u \rangle\}. \label{sca-sub-lip}
\end{equation*}
\end{enumerate}
\end{Definition}

Clarke's subdifferential generalizes the well known concept of Fenchel subdifferential of a convex functional to the locally Lipschitz setting, see Proposition \ref{prop:clarke properties} \eqref{item:clark convex} below. For a convex functional $f$ and the point $\bar{x} \in \dom f,$ this is the set

$$ \partial f(\bar{x}):= \{x^* \in X^* \mid \forall\; x \in X: \langle x^*,x-\bar{x}\rangle \leq f(x) - f(\bar{x})\}.$$

We continue by defining Clarke's tangent and normal cones to a set.

\begin{Definition}
Let $A \subseteq X$ be nonempty and let $\bar{x}\in A$. 

\begin{enumerate}
\item The Clarke tangent cone to $A$ at $\bar{x}$ is the set defined as 
$$ T(A,\bar{x}):=\left\{u\in X \bigl\vert \begin{array}{ll}
   \forall\; \{x_k\}_{k \geq 1} \subseteq A,\; \{t_k\}_{k\geq 1} \subseteq \R \mbox{ with } x_k \to \bar{x}, \;  t_k \downarrow 0: \\
     \exists \; \{u_k\}_{k\geq 1} \textrm{ such that  } u_k \to u \textrm{ and } x_k +t_ku_k \in A \;\forall\;k\end{array} \right\}.$$  A vector $u\in T(A,\bar{x})$ is called a tangent to $A$ at $\bar{x}$.

\item The Clarke normal cone of $A$ at $\bar{x}$ is the set defined by 
\[
N_C(A,\bar{x}):=\{x^{*}\in X^{*}\mid \forall\; u\in T(A,\bar{x}): \langle x^*,u \rangle\leq 0 \}.
\]
\end{enumerate}
\end{Definition}
We can now define the coderivative of a set-valued mapping in the sense of Clark.
\begin{Definition}
Let $F:X \rightrightarrows Y$ be a given set-valued mapping and let $(\bar{x},\bar{y}) \in \gph F.$ The Clarke coderivative of $F$ at  $(\bar{x},\bar{y})$ is the set-valued mapping $D^*_C F(\bar{x},\bar{y}): Y^* \rightrightarrows X^*$ defined as

$$D^*_C F(\bar{x},\bar{y})(y^*):= \{x^* \in X^* \mid (x^*,-y^*) \in N_C(\gph F,(\bar{x},\bar{y}))\}.$$ 
\end{Definition}

We close this section with the following proposition, that collects some properties of Clarke's subdifferential that are used in the main results, see \cite{Clarke1} for the proofs.

\begin{Proposition}\label{prop:clarke properties}  Let $X$ be a Banach space and $f,g:X \rightarrow \overline{\R}$ be locally Lipschitz at $\bar{x} \in X.$ The following statements hold:

\begin{enumerate}

\item \label{item:clarke w*compact subdif} The set $\partial^{\circ} f(\bar{x})$ is nonempty, convex, and $w^*$- compact (that is, compact w.r.t. the $w^*$- topology).

\item \label{item:clarke sum rule} $\partial^\circ (f  + g) (\bar{x}) \subseteq  \partial^\circ f(\bar{x}) + \partial^\circ g(\bar{x}).$ Equality holds if $f $ and $g$ are convex.

\item \label{item: clarke directional deriv = support} $f^\circ (\bar{x}, \cdot)= \sigma_{\partial^\circ f(\bar{x})}(\cdot).$

\item \label{item:clarke fermat} If $\bar{x}$ is a local minimum of $f,$ then $0 \in \partial^\circ f(\bar{x}).$

\item \label{item:clark convex} If $f$ is convex, then $\partial^\circ f(\bar{x}) = \partial f(\bar{x}).$

\end{enumerate}

\end{Proposition}

\section{Main Result}

The basic setting in which we derive the results of this section is the following:

\begin{Assumption}\label{ass marginal}
Let $X$ and $Y$ be Banach spaces and $F:X \rightrightarrows Y$ be a set-valued mapping with $F(x)\neq\emptyset$ for every $x\in X.$ Furthermore, let $f:X\times Y \to \overline{\R}$ be a given extended real valued functional, and consider the associated marginal function $\varphi:X\to \overline{\R}$ given in \eqref{eq: marginal function}. In addition, consider the set-valued solution mapping $S:X\rightrightarrows Y$ defined as 

\begin{equation*}\label{eq: solution map}
S(x):= \{y\in F(x): f(x,y)= \varphi(x)\}.
\end{equation*} 

\end{Assumption}

We start by proving some lemmata that contribute to the main theorem. The assertion in the next lemma  is proved using Ekeland's variational principle,  see \cite{Eke74,eke79,Schirotzek2007}.

\begin{Lemma}\label{lem: ekeland's lemma}

Let $g: Y \to \mathbb{R}$ be a given functional, and suppose that $g$ is convex, continuous, and bounded from below. Then, there are sequences $\{v_k\}_{k\geq 1} \subseteq Y$ and $\{y_k^*\}_{k\geq 1}\subseteq Y^*$ such that:
\begin{enumerate}
\item $g(v_k) \to \inf \limits_{y\in Y} g(y),$

\item $y_k^*\in \partial g(v_k),$

\item $|\langle y_k^*, v_k \rangle |\to 0,$

\item $y_k^* \to 0.$
\end{enumerate}

\end{Lemma}
\begin{proof}
Let $\{y_k\}_{k\geq 1}$ be a minimizing sequence such that 

$$g(y_k) < \inf \limits_{y\in Y} g(y) +\frac{1}{k}.$$ By Ekeland's variational principle \cite[Theorem 8.2.4]{Schirotzek2007}, for every $\lambda_k >0$ we can find $v_k \in Y$ such that

\begin{itemize}
\item[(a)] $\|y_k-v_k\|\leq \lambda_k,$

\item[(b)] $g(v_k) \leq \inf \limits_{y\in Y} g(y) + \frac{2}{k},$ and

\item[(c)] $v_k$ is a minimizer of the functional 

$$g_k(\cdot):= g(\cdot)+ \frac{1}{k\lambda_k}\|\cdot -v_k\|.$$
\end{itemize}

Since $g$ is a continuous convex functional, it follows that $g_k$ is also convex and continuous. Hence, from (c) and items \eqref{item:clarke fermat}, \eqref{item:clark convex}, and \eqref{item:clarke sum rule} in Proposition \ref{prop:clarke properties}, we deduce that

\begin{equation}\label{eke1}
0\in \partial g_k(v_k)= \partial g(v_k)+ \frac{1}{k\lambda_k}\mathbb{B}^*.
\end{equation}

From \eqref{eke1}, we then get the existence of $y_k^* \in \partial g(v_k)$ with 
 
\begin{equation}\label{eq: y_k^* upper bound}
\|y_k^*\|_* \leq \frac{1}{k \lambda_k}.
\end{equation} On the other hand, from (a), we have in particular that

\begin{equation}\label{eq: v_k upper bound}
\|v_k\| \leq \|y_k\| + \lambda_k.
\end{equation} The inequalities \eqref{eq: y_k^* upper bound}  and \eqref{eq: v_k upper bound} now imply 

\begin{equation}\label{eq: y_k*, v_k upper bound }
|\langle y_k^* , v_k \rangle| \leq \|y^*_k\|_* \|v_k\|\leq  \frac{\lambda_k + \|y_k\|}{k \lambda_k}.
\end{equation}

Since  $v_k$ and $y_k^*$ depend on $\lambda_k,$ we  next show that $\lambda_k$ can be chosen such that the properties $(i)- (iv)$ in the statement are true. Indeed, independently of the choice of the sequence $\{\lambda_k\}_{k\geq 1}$, we get properties $(i)$ and $(ii)$ from (b) and \eqref{eke1}, respectively. In order to guarantee  $(iii)$ and $(iv)$, we set $\lambda_k:= \|y_k\| + 1$ for every $k \in \N.$ Then, from \eqref{eq: y_k^* upper bound}, we find that

$$\|y_k^*\|_* \leq \frac{1}{k(\|y_k\| + 1)}  \leq \frac{1}{k} \to 0,$$ which proves $(iii).$ Furthermore, from \eqref{eq: y_k*, v_k upper bound } it follows that

$$|\langle y_k^* , v_k \rangle| \leq  \frac{\lambda_k + \|y_k\|}{k \lambda_k} = \frac{2 \|y_k\| + 1}{k (\|y_k\| + 1)} \leq \frac{2}{k} \to 0,$$ and thus $(iv)$ is fulfilled. The proof is complete.

% \begin{itemize}
% \item \textbf{Case 1: } The sequence $\{y_k\}_{k\geq 1}$ is bounded.

% Then, we can choose $\lambda_k=1$ for every $k\in \mathbb{N}.$ If $\alpha > 0$ is such that $\|y_k\|\leq \alpha,$ then, in \eqref{eq: y_k*, v_k upper bound } we obtain 

% $$|\langle y_k^* , v_k \rangle| \leq  \frac{1 + \alpha}{k}  \to 0.$$  This shows $(iii).$ In addition, from \eqref{eq: y_k^* upper bound} we also get $\|y_k^*\|_* \leq \frac{1}{k } \to 0.$ This proves $(iv).$

% \item \textbf{Case 2:} The sequence $\{y_k\}_{k\geq 1}$ is not bounded.
% Then, without loss of generality we assume that $\|y_k\| \to +\infty.$ Set now $\lambda_k = \|y_k\|.$ Then,  from \eqref{eq: y_k*, v_k upper bound } we get 

% $$|\langle y_k^* , v_k \rangle| \leq   \frac{2}{k}  \to 0,$$ and hence $(iii)$ is fulfilled. Moreover, from  \eqref{eq: y_k^* upper bound} we obtain $\|y_k^*\|_* \leq \frac{1}{k \|y_k\|} \to 0,$ which verifies $(iv).$ 

% \end{itemize}

\end{proof}

%\begin{Proposition}\cite[Proposition 4.5.2]{Schirotzek2007}\label{prop: maximum rule convex schirotzek}
%Let $Z$ be a normed space and let $A$ be a compact Hausdorff space. For any $a \in A,$ let $g_a : Z \to \R$ be convex on $Z$ and continuous at $\bar{z} \in Z.$ Define the functional $g: Z \to \R$ as $g(z):= \sup\limits_{a \in A} g_a(z)$ and consider the set $$A(\bar{z}):= \{a \in A \mid  g(\bar{z})= g_a(\bar{z})\}.$$ Assume further that there exists a neighborhood $W$ of $\bar{z}$ such that for every $z \in W$, the functional $a \to g_a(z)$ is upper semicontinuous on $A.$ Then, 
%
%$$\partial g (\bar{z})= \overline{\operatorname{conv}}^*\left( \bigcup_{a \in A(\bar{z})} \partial g_a(\bar{z})\right).$$
%
%
%\end{Proposition}
%
%

Before stating the next lemma, we establish some notation. For a set $G \subseteq X^* \times Y^*,$ the projection of $G$ onto $X^*$ is defined as
 
 $$G_{X^*}:= \{x^* \in X^* \mid \; \exists \; y^* \in Y^*: (x^*,y^*) \in G\}.$$ 
The projection of $G$ onto $Y^*$, denoted by $G_{Y^*},$ is defined similarly. For $(\overline{u}, \overline{v}) \in X \times Y$, we put
 
$$
G^{(\bar{u},\bar{v})}:= \{(x^*,y^*) \in  G \mid\;  \langle (x^*,y^*),(\bar{u},\bar{v})\rangle=\sigma_G((\bar{u},\bar{v})) \},
$$ 
and hence,

$$
G^{(\bar{u},\bar{v})}_{Y^*}=\{y^* \in Y^* \mid \exists\; x^*\in X^* \textrm{ such that } (x^*,y^*) \in G^{(\bar{u},\bar{v})}\}. 
$$ 
It is easy to see that the sets $G^{(\bar{u},\bar{v})}$  and $G^{(\bar{u},\bar{v})}_{Y^*}$ are convex and $w^*$- closed if $G$ is convex and $w^*$- closed.
 
\begin{Lemma}\label{lem: subdif support projection function}

Let $G$ be a $w^*$- compact convex subset of $X^*\times Y^*.$ Fix $\bar{u} \in X$ and consider the functional $g: Y \to \mathbb{R}$ defined as follows:

\begin{equation}\label{eq: support projection function}
 g(v):= \sigma_G(\bar{u},v).
\end{equation} Then, $g$ is a continuous convex functional satisfying 

$$\forall \; \bar{v} \in Y: \partial g(\bar{v})= G^{(\bar{u},\bar{v})}_{Y^*}.$$

\end{Lemma}

\begin{proof} The statement will be a consequence of \cite[Proposition 4.5.2]{Schirotzek2007}. Indeed, by definition, $G$ with the $w^*$- topology is a compact Hausdorff space. Now consider, for $(x^*,y^*) \in G,$ the functional $g_{(x^*,y^*)}: Y \to  \R$ defined by $g_{(x^*,y^*)}(v):= \langle y^*,v\rangle + \langle x^*,\bar{u}\rangle. $ Obviously, every functional of this type is convex and continuous at $\bar{v}.$ Furthermore, by definition, we have  $ g(v)=  \sup\limits_{(x^*,y^*) \in G} g_{(x^*,y^*)}(v) $ for every $v \in Y.$ Moreover, because of the definition of the $w^*$- topology, the map $(x^*,y^*) \to g_{(x^*,y^*)}(v)$ is continuous (and hence upper semicontinuous) for any $v\in Y.$  

On the other hand, note that $G^{(\bar{u},\bar{v})}= \bigg\{(x^*,y^*) \in G \mid  g_{(x^*,y^*)}(\bar{v}) = g(\bar{v})\bigg\}$ and that $\partial g_{(x^*,y^*)}(\bar{v})= \{y^*\}.$ Taking this into account and applying \cite[Proposition 4.5.2]{Schirotzek2007}, we obtain 

\begin{eqnarray*}
\partial g(\bar{v}) & = & \overline{\operatorname{conv}}^*\left( \bigcup_{(x^*,y^*) \in G^{(\bar{u},\bar{v})}} \partial g_{(x^*,y^*)}(\bar{v})\right)\\
                    & = & \overline{\operatorname{conv}}^*\left( \bigcup_{(x^*,y^*) \in G^{(\bar{u},\bar{v})}} \{y^*\}\right)\\
                    & = & \overline{\operatorname{conv}}^*\left( G^{(\bar{u},\bar{v})}_{Y^*}\right).
\end{eqnarray*} Thus, the statement follows from the fact that $G^{(\bar{u},\bar{v})}_{Y^*}$ is $w^*$- closed and convex.
\end{proof}

To proceed, we need the notions of inner semicompactness and closedness of a set-valued mapping. The inner semicompactness and closedness of the solution mapping $S$ from  Assumption \ref{ass marginal} are standard hypotheses when deriving estimates for subdifferentials of marginal functions, see for example \cite[Theorem 1.108 and Theorem 3.38]{Mordukhovich1}.

\begin{Definition}\label{def:propsetvmap} Let  $\bar{x}\in\dom S$. We say that:

\begin{enumerate}
\item  $S$ is inner semicompact at $\bar{x},$ if $\bar{x} \in \Int (\dom S)$ and, for every sequence $x_k \rightarrow \bar{x},$ there is a sequence $y_k \in S(x_k)$ that contains a convergent subsequence as $k \rightarrow \infty$. 

\item $S$ is closed at $\bar{x},$ if for any sequence $\{(x_k,y_k)\}_{k\geq 1} \subseteq \gph S$ with $(x_k,y_k) \to (\bar{x},\bar{y}),$  we have $(\bar{x},\bar{y})\in \gph S.$
\end{enumerate}

\end{Definition}

\begin{Remark}
By Weierstrass's Theorem \cite[Theorem 2.3]{jahn2006} it is straightforward to verify that, in our context, the set-valued mapping $S$ is inner semicompact at $\bar{x}$ if the following assumptions are fulfilled:
\begin{enumerate}
    \item $Y$ is finite dimensional,
    \item there exists a neighborhood $U$ of $\bar{x}$ such that:
\begin{itemize}
     \item $f$ is lower semicontinuous (l.s.c) on $U$,
     \item  $F(x)$ is compact for every $x \in U,$
     \item the set $F[U]$ is bounded.
\end{itemize}
\end{enumerate}

\end{Remark}
\begin{Remark}\label{rem: closedness of S}
A sufficient condition for the closedness of $S$ at $\bar{x}$ is that 

\begin{enumerate}
\item $\varphi$ is continuous at $\bar{x},$

\item $f$ is l.s.c on $\{\bar{x}\}\times F(\bar{x}),$

\item $F$ is closed at $\bar{x}.$
\end{enumerate}

Indeed, let $\{(x_k,y_k)\}_{k\geq 1} \subseteq \gph S$ be convergent to $(\bar{x},\bar{y}).$ Then, for every $k\in  \mathbb{N},$ we have $f(x_k,y_k) = \varphi(x_k).$ Because $\varphi$ is continuous at $\bar{x},$ it follows that $f(x_k,y_k) \to \varphi(\bar{x}).$ Since  $S(x_k) \subseteq F(x_k)$ for every $k \in \mathbb{N}$ and $F$ is closed at $\bar{x},$ we obtain $\bar{y} \in F(\bar{x}).$ Hence, we also have that $f$ is l.s.c at  $(\bar{x},\bar{y}).$ From this, we get

$$f(\bar{x},\bar{y})\leq \liminf_{k\to \infty} f(x_k,y_k) = \varphi(\bar{x}),$$ which implies $\bar{y} \in S(\bar{x}).$
\end{Remark}

In the following lemma, we analyze Clarke's subdifferential of the marginal function $\varphi$ in the unconstrained case $F(\cdot):= Y$. This is a weaker version of our main result that will be derived in Theorem \ref{thm: basic subdif marginal functions theorem}.

\begin{Lemma}\label{lem: subdif marginal lipsch unconstrained}

Let $\bar{x} \in X$  and suppose that:

\begin{enumerate}

\item $ \forall \; x \in X: F(x):= Y,$
\item $f$ is locally Lipschitz on $\{\bar{x}\}\times S(\bar{x}),$

\item the associated marginal function $\varphi$ is locally Lipschitz at $\bar{x}\in X,$ 

\item the set-valued solution mapping $S$ is inner semicompact and closed at $\bar{x}.$
\end{enumerate}
Then,

\begin{equation}\label{eq: unconstrained marginal}
\partial^{\circ} \varphi(\bar{x})\subseteq \overline{\operatorname{conv}}^*\left( \bigcup_{\bar{y}\in S(\bar{x})}   \big\{x^* \in X^*: (x^*,0) \in \partial^{\circ} f(\bar{x},\bar{y})\big\}    \right).
\end{equation}
 
\end{Lemma}

\begin{proof}

Consider the set-valued mapping $\tilde{S}:X \rightrightarrows Y$ defined as $$\tilde{S}(x):=  \{y\in S(x): \; \partial^{\circ} f(x,y) \cap \left( X^*\times \{0\}\right) \neq \emptyset\}.$$ Then, it suffices to take the inclusion  \eqref{eq: unconstrained marginal} over $\bar{y}\in \tilde{S}(\bar{x})$ in the proof.  We need the following claims:

\begin{itemize}
 \item \textbf{Claim 1:} The  inequality below  holds: 
 
 $$\varphi^\circ(\bar{x},\cdot) \leq \sup_{\bar{y}\in S(\bar{x})} \inf_{v\in Y} f^\circ((\bar{x},\bar{y}),(\cdot,v)).$$

Indeed, fix any $u\in X.$ Then, we can find sequences $\{x_k\}_{k\geq 1} \subseteq  X$ and $\{t_k\}_{k\geq 1}\downarrow 0$ such that $x_k\to \bar{x}$ and 

\begin{equation}\label{eq: claim1_1}
\varphi^\circ(\bar{x},u)= \lim_{k\to \infty} \frac{\varphi(x_k+t_k u)-\varphi(x_k)}{t_k}.
\end{equation} Since $S$ is inner semicompact at $\bar{x},$ there exists a sequence $\{y_k\}_{k\geq 1} \subseteq Y$ such that $y_k \in S(x_k)$ for every $k\in \mathbb{N}$ and $\{y_k\}_{k\geq 1}$ contains a convergent subsequence. Without loss of generality, we  assume that $y_k \to \bar{y}$ for some $\bar{y}\in Y.$ Since $S$ is closed at $\bar{x},$ it follows that $\bar{y}\in S(\bar{x}).$ Then, taking into account \eqref{eq: claim1_1} and the definition of $\varphi,$ for any $v\in Y$ we have:
\begin{eqnarray*}
\varphi^\circ(\bar{x},u)   &  \leq & \limsup_{k\to \infty} \frac{f(x_k +t_k u, y_k +t_k v)-f(x_k,y_k)}{t_k}\\
                           & \leq & \limsup_{(x,y)\to (\bar{x},\bar{y}), \; t\downarrow 0} \frac{f(x +tu, y+tv)-f(x,y)}{t}\\
                           &  =  &  f^\circ((\bar{x},\bar{y}),(u,v)).
\end{eqnarray*} Taking now the infimum over $v\in Y$ in the above inequality, we obtain the desired result.

\item \textbf{Claim 2:} $\tilde{S}(\bar{x})\neq \emptyset.$ 

We show that the element  $\bar{y}\in S(\bar{x})$ constructed in the proof of Claim 1 satisfies $\bar{y}\in \tilde{S}(\bar{x}).$ Indeed, fix any $\bar{u}\in X.$ Then, from Claim 1, we get 

\begin{equation}\label{eq:claim1fixedu}
-\infty < \varphi^\circ(\bar{x},\bar{u})\leq \inf_{v\in Y}f^\circ((\bar{x},\bar{y}),(\bar{u},v)).
\end{equation} Set  $G:= \partial^{\circ} f(\bar{x},\bar{y})$ and consider the sets $G_{X^*}, G_{Y^*}.$ Then, by Proposition \ref{prop:clarke properties} \eqref{item:clarke w*compact subdif}, we know that $G$ is nonempty, convex, and $w^*$-  compact. It is then straightforward to verify that $G_{X^*}$ and $G_{Y^*}$ are also nonempty, convex and $w^*$- compact sets. Taking into account now \eqref{eq:claim1fixedu} and Proposition \ref{prop:clarke properties} \eqref{item: clarke directional deriv = support}, we deduce that
$$-\infty < \inf\limits_{v\in Y} \sigma_G(\bar{u},v)\leq \sigma_{G_{X^*}}(\bar{u}) + \inf_{v\in Y}\sigma_{G_{Y^*}}(v).$$ Thus, in particular, 

\begin{equation}\label{eq: claim 2_1}
\inf \limits_{v\in Y}\sigma_{G_{Y^*}}(v) > -\infty.
\end{equation} We now show that $0\in G_{Y^*}.$ Indeed, otherwise we could strongly separate the sets $\{0\}$ and $G_{Y^*}$ (see, e.g., \cite[Theorem 2.2.8.]{GopRiaTamZal:03}) to obtain a vector $\bar{v}\in Y$ such that $$0 = \langle 0, \bar{v} \rangle > \sigma_{G_{Y^*}}(\bar{v}).$$ Hence, taking into account the positive homogeneity of  $\sigma_{G_{Y^*}},$ we obtain

$$-\infty=\lim_{\lambda \to \infty} \sigma_{G_{Y^*}}(\lambda \bar{v}),$$ a contradiction to \eqref{eq: claim 2_1}.

By the definition of $G_{Y^*},$ we conclude that there exists $x^* \in X^*$ such that $(x^*,0)\in G,$ or equivalently, $\bar{y}\in \tilde{S}(\bar{x}).$

\item \textbf{Claim 3:} For $\bar{y} \in \tilde{S}(\bar{x})$  and $\bar{u}\in X,$ we have 

\begin{equation}\label{eq: claim 3}
\sup_{(x^*,0) \in \partial^{\circ} f(\bar{x},\bar{y})} \langle x^*,\bar{u}\rangle= \inf_{v\in Y}f^\circ ((\bar{x},\bar{y}),(\bar{u},v)).
\end{equation}

In order to prove \eqref{eq: claim 3}, let us use the notation of Claim 2 and the function $g$ defined  by \eqref{eq: support projection function}. By Lemma \ref{lem: subdif support projection function}, we have

$$\forall\; v\in Y: \partial g(v)=  G^{(\bar{u},v)}_{Y^*}.$$ 
Furthermore, because $g$ is continuous, convex and bounded from below, it is possible to apply Lemma \ref{lem: ekeland's lemma}. Therefore, we obtain the existence of sequences $\{v_k \}\subseteq Y$ and $\{y_k^*\}_{k\geq 1}\subseteq Y^*$ such that for every $k\in \mathbb{N}$:
\begin{itemize}
\item[(a)] $g(v_k) \to \inf \limits_{y\in Y} g(y),$

\item[(b)] $y_k^*\in G^{(\bar{u},v_k)}_{Y^*},$

\item[(c)] $|\langle y_k^*, v_k \rangle |\to 0,$

\item[(d)] $y_k^* \to 0.$
\end{itemize} In particular, from (b) we conclude the existence of a sequence $\{x_k^*\}_{k\geq 1}\subseteq G_{X^*}$ such that for every $k\in \mathbb{N}$ the inclusion $(x_k^*,y_k^*) \in G$ holds and

\begin{equation}\label{eq: magic happens}
g(v_k) = \langle x_k^*, \bar{u} \rangle + \langle y_k^*, v_k  \rangle.
\end{equation} Since $G$ is $w^*$- compact, the net $\{(x_k^*,y_k^*)\}_{k\geq 1}$ has a convergent subnet. Without loss of generality, let $(x_k^*,y_k^*) \overset{w^*}{\to} (\bar{x}^*,\bar{y}^*)\in G.$ By (d), we get in particular $\bar{y}^*=0$ and $(\bar{x}^*,0) \in G.$ Taking into account (a) and (c), and taking the limit in \eqref{eq: magic happens}, we obtain

\begin{equation}\label{eq: claim 3 const}
\langle \bar{x}^*, \bar{u} \rangle = \inf_{v\in Y} g(v).
\end{equation} This means that

\begin{eqnarray*}
\sup_{(x^*,0) \in \partial^{\circ} f(\bar{x},\bar{y})} \langle x^*,\bar{u}\rangle  \geq   \langle \bar{x}^*, \bar{u} \rangle & \overset{\eqref{eq: claim 3 const}}{=} & \inf_{v\in Y} g(v)\\
    & \overset{\eqref{eq: support projection function}}{=} & \inf_{v\in Y}\sigma_G(\bar{u},v)\\
    & \overset{(\textrm{Proposition  }\ref{prop:clarke properties} \;\eqref{item: clarke directional deriv = support})}{=} & \inf_{v\in Y}f^\circ ((\bar{x},\bar{y}),(\bar{u},v)).    
\end{eqnarray*}

In order to see the validity of the reverse inequality  note that, if $(x^*,0) \in G,$ then from the definition of $\sigma_G$ we obtain 

$$\forall \; (u,v) \in X\times Y:\; \langle x^*, u \rangle \leq \sigma_G(u,v).$$ Fixing $u= \bar{u}$ and taking the infimum over $v\in Y,$ the desired inequality is obtained.

\end{itemize}

Now, assume that $$\bar{x}^* \notin  \overline{\operatorname{conv}}^*\left( \bigcup_{\bar{y}\in \tilde{S}(\bar{x})} \{x^* \in X^*: (x^*,0) \in \partial^{\circ} f(\bar{x},\bar{y})\} \right).$$ By the classical strong separation theorem (see, e.g., \cite[Theorem 2.2.8.]{GopRiaTamZal:03}), we  then find  $\bar{u} \in X$ such that 

\begin{equation}\label{eq: intermediate ineq subdif clarke}
\langle \bar{x}^*,\bar{u} \rangle > \sup_{\bar{y}\in \tilde{S}(\bar{x})} \sup_{(x^*,0)\in \partial^\circ f(\bar{x},\bar{y})}\langle x^*, \bar{u} \rangle.
\end{equation} However, because of Claim 3, we know that \eqref{eq: claim 3} holds. Putting this back into \eqref{eq: intermediate ineq subdif clarke} and using Claim 1 we obtain

$$\langle \bar{x}^*,\bar{u} \rangle > \sup_{\bar{y}\in \tilde{S}(\bar{x})} \inf_{v\in Y}f^\circ ((\bar{x},\bar{y}),(\bar{u},v))\geq \varphi^\circ(\bar{x},\bar{u}).$$ This means that $\bar{x}^* \notin \partial^\circ \varphi(\bar{x}),$ and hence the desired inclusion holds. The proof is complete.
\end{proof}

\begin{Remark} Suppose that $X = \R^n,\; Y = \R^m,$ and that $$f(x,y) = \max \{f_i(x,y) \mid i = 1, \ldots, s\}$$ for some continuously differentiable functions $f_i: \R^n\times \R^m \rightarrow \R,\; i = 1,\ldots,s,$ and some $s\in \N.$ Furthermore, consider the set-valued mapping $I:\R^n\times \R^m \rightrightarrows \{1,\ldots,s\}$ given by $I(x,y) := \{i \in \{1,\ldots,s\} : f_i(x,y) = f(x,y)\}.$ Then, the upper estimate of $\partial^\circ \varphi (\bar{x})$ given in \eqref{eq: unconstrained marginal} is the closure of the one provided in \cite[Theorem 4]{bot2012}, where it is shown under different assumptions that
\begin{equation*}
    \partial^\circ \varphi (\bar{x}) \subseteq \conv \left(\bigcup_{\bar{y} \in S(\bar{x})} \left\{ \sum_{i \in I(\bar{x}, \bar{y})} \mu_i \nabla_xf_i(\bar{x},\bar{y}) : \begin{matrix}\mu_i \geq 0\; \forall\; i \in I(\bar{x},\bar{y}),\\ \sum\limits_{i \in I(\bar{x},\bar{y})}\mu_i = 1,\\\sum\limits_{i \in I(\bar{x},\bar{y})}\mu_i \nabla_y f_i(\bar{x},\bar{y}) = 0
    \end{matrix}    \right\}\right).
\end{equation*} Indeed, because of the particular structure of $f,$ in this case we can apply \cite[Proposition 2.3.12]{Clarke1} to obtain 
$$ \forall\; (x,y) \in \R^n \times \R^m:  \partial^\circ f(x,y) = \conv \{\nabla f_i(x,y) : i \in I(x,y) \}.$$ Thus, in \eqref{eq: unconstrained marginal}, we have that $(x^*,0) \in \partial^\circ f(\bar{x},\bar{y})$ if an only if for each $i \in I(\bar{x},\bar{y})$ there exists $\mu_i \geq 0$ such that $\sum_{i \in I(\bar{x},\bar{y})}\mu_i = 1, \sum_{i \in I(\bar{x},\bar{y})}\mu_i \nabla_y f_i(\bar{x},\bar{y}) = 0,$ and  $x^* =\sum_{i \in I(\bar{x}, \bar{y})} \mu_i \nabla_xf_i(\bar{x},\bar{y}).$ This implies the statement.
\end{Remark}

We can now state the main result of this paper.

\begin{Theorem}\label{thm: basic subdif marginal functions theorem}
Let  $\bar{x} \in X $ and suppose that:

\begin{enumerate}

\item $F$ satisfies the Aubin property at every point of the set $\{\bar{x}\}\times S(\bar{x}),$

\item there is a neighborhood $U$ of $\bar{x}$ and a constant $L_1>0$ such that, for every $x\in U:$

\begin{itemize}
\item $f(x,\cdot)$ is Lipschitz on $Y$ with constant $L_1,$ 

\item $F(x)$ is closed.

\end{itemize}

\item $f$ is locally Lipschitz on $\{\bar{x}\}\times S(\bar{x}),$

\item the associated marginal function $\varphi$ is Lipschitz around $\bar{x}\in X,$ 

\item the set-valued solution mapping $S$ is inner semicompact and closed at $\bar{x}.$

\end{enumerate} Then,
\begin{equation}\label{eq: upper estimate marginal}
\partial^{\circ} \varphi(\bar{x})\subseteq \overline{\operatorname{conv}}^*\left( \bigcup_{\underset{(x^*,y^*)\in \partial^{\circ} f(\bar{x},\bar{y})}{\bar{y}\in S(\bar{x})}}  \bigg[ x^*+ D_C^* F(\bar{x},\bar{y})(y^*) \bigg] \right).
\end{equation}

\end{Theorem}
\begin{proof} Since $S$ is inner semicompact at $\bar{x},$ in particular we have $\bar{x} \in \Int (\dom S).$ Hence, without loss of generality, we can assume that $U \subseteq \Int (\dom S).$ It follows that the assumptions of \cite[Proposition 2.4.3]{Clarke1} are fulfilled. Then, for any $L' > L_1$ and every $x\in U,$ 
	
	$$ \varphi(x)=\inf_{y\in Y} \bigg \{f(x,y)+ L' \Delta_F(x,y)\bigg\}$$  and
	$$S(x) = \big\{y \in Y \mid f(x,y) + L' \Delta_F(x,y)= \varphi(x)\big\},$$
where the function $\Delta_F: X\times Y\rightarrow \R$ given by $\Delta_F(x,y):=\inf\limits_{u\in F(x)}\|y-u \|$. We consider $f_1: X\times Y \to \R$ defined as 

$$f_1(x,y):= f(x,y)+ L' \Delta_F(x,y).$$
Therefore,

$$\forall \; x\in U: \;\varphi(x)=\inf_{y\in Y} f_1(x,y).$$ Since $F$ satisfies the Aubin property at every point of the set $\{\bar{x}\}\times S(\bar{x}),$ according to \cite[Theorem 1.41]{Mordukhovich1}, we have that $\Delta_F$ is locally Lipschitz at any point $(\bar{x},\bar{y}) \in \{\bar{x}\} \times S(\bar{x}).$ This, together with assumption $(iii),$ implies that $f_1$ is also  locally Lipschitz at every point of the set $\{\bar{x}\} \times S(\bar{x}).$ Applying now Lemma \ref{lem: subdif marginal lipsch unconstrained}, we get 

\begin{equation}\label{eq: basic subdif marginal functions theorem 1}
\partial^{\circ} \varphi(\bar{x})\subseteq \overline{\operatorname{conv}}^*\left( \bigcup_{\bar{y}\in S(\bar{x})} \left\{ x^* \in X^* \mid (x^*,0) \in \partial^{\circ} f_1(\bar{x},\bar{y})\right\} \right).
\end{equation} By  Proposition \ref{prop:clarke properties} \eqref{item:clarke sum rule} and \cite[Corollary 1.4 $(ii)$]{Thibault1991}, we have

\begin{equation*}
\partial^{\circ} f_1(\bar{x},\bar{y}) \subseteq \partial^{\circ} f(\bar{x},\bar{y}) +\partial^{\circ} \Delta_F(\bar{x},\bar{y})\subseteq \partial^{\circ} f(\bar{x},\bar{y})+ N_C(\gph F,(\bar{x},\bar{y})).
\end{equation*} Taking into account this with \eqref{eq: basic subdif marginal functions theorem 1}, we get

$$\partial^{\circ} \varphi(\bar{x})\subseteq \overline{\operatorname{conv}}^*\left( \bigcup_{\bar{y}\in S(\bar{x})} \left\{ x^* \in X^* \mid (x^*,0) \in  \partial^{\circ} f(\bar{x},\bar{y})+ N_C(\gph F,(\bar{x},\bar{y}))\right\} \right),$$ which is equivalent to our statement.
\end{proof}

\begin{Remark} 

If the spaces $X$ and $Y$ are Asplund, then an upper estimate of Clarke's subdifferential of $\varphi$ can be obtained by means of the result  in \cite[Theorem 7]{MNY2009} (see also \cite{Thibault1991}) and the fact that $\partial^\circ \varphi(\bar{x}) = \overline{\operatorname{conv}}^* \left( \partial_M \varphi (\bar{x}) \right),$ where  $\partial_M$ stands for Mordukhovich's (limiting) subdifferential, see \cite[Theorem 3.57 $(ii)$]{Mordukhovich1}. In that case, one can show that

$$\partial^{\circ} \varphi(\bar{x})\subseteq \overline{\operatorname{conv}}^*\left( \bigcup_{\bar{y}\in S(\bar{x})} \left\{ x^* \in X^* \mid (x^*,0) \in  \partial^{\circ} f(\bar{x},\bar{y})+ N_M(\gph F,(\bar{x},\bar{y}))\right\} \right),$$ where $N_M(\gph F,(\bar{x},\bar{y}))$ denotes the limiting normal cone of $\gph F$ at $(\bar{x},\bar{y})$. Moreover, according to \cite[Theorem 3.57 $(i)$]{Mordukhovich1}, we also have  $$N_C(\gph F,(\bar{x},\bar{y})) = \overline{\operatorname{conv}}^* \left(N_M(\gph F,(\bar{x},\bar{y}))\right).$$ It then follows that an upper estimate obtained in this way is significantly smaller than the one derived by us in Theorem \ref{thm: basic subdif marginal functions theorem}. Thus, our results are more relevant in the case in which the Asplund assumption on the spaces is not satisfied.

\end{Remark}
We conclude the section with an example that illustrates an additional drawback of the upper estimate  \eqref{eq: upper estimate marginal} in Theorem \ref{thm: basic subdif marginal functions theorem}. Specifically, we show that, depending on the representation of $\varphi$  by the set-valued mapping $F$, the estimates obtained can be both trivial and exact.

\begin{Example}\label{ex: dpendency subdiff}
Set $X = Y = \R,$ and consider $F_1, F_2: \R \rightrightarrows \R, f: \R^2 \to \R$ defined respectively as 

$$\forall\; (x,y) \in \R^2 : \;F_1(x) = \{|x|\}, \;\; F_2(x) = F_1(x) + [0,1],\;\; f(x,y) = y.$$ Then, we have $|\cdot| = \varphi(\cdot) = \inf\limits_{y\in F_1(\cdot)} f(\cdot,y) = \inf\limits_{y\in F_2(\cdot)} f(\cdot,y).$ Let us then analyze the upper estimate \eqref{eq: upper estimate marginal} for each case  when $\bar{x} = 0$. 

\begin{enumerate}
\item \textbf{Case 1:} $F = F_1.$

Since $S(0) = \{0\}$ and $\partial^\circ f(0,0) = \left\{\begin{pmatrix} 0 \\ 1 \end{pmatrix} \right\},$ the upper estimate is exactly 

$$U_1: = \overline{\operatorname{conv}}^*\left(\bigcup_{(x^*,y^*) \in \{(0,1)\}} [x^* + D^*_C F_1(0,0)(y^*)] \right) = D^*_C F_1(0,0)(1).$$ Furthermore, it is well known that $N_C(\gph |\cdot|, (0,0)) = \R^2,$ so that we actually have $U_1  = \R.$

\item \textbf{Case 2:} $F = F_2.$

Similarly to the previous case, we find that the upper estimate is  

$$U_2 := D^*_C F_2(0,0)(1) = [-1,1] = \partial^\circ \varphi(0),$$ and hence the equality holds.

\end{enumerate}

\end{Example}

%\section{Application in Set Optimization}\label{Section:application}

%%%%%%%%%%%%%%%%%
%% Conclusions %%
%%%%%%%%%%%%%%%%%
\section{Conclusions}
\label{section:Conclusions}

In our paper, we derived an upper estimate of Clarke’s subdifferential of marginal functions in Banach spaces. The results are mostly of interest in the case in which the underlying spaces are Banach, but do not necessarily satisfy the Asplund condition.  Using our results, it is possible to derive necessary optimality conditions for solutions of set-valued optimization problems in the setting of general Banach spaces analogously to the procedure in \cite{bouzaquintanatuantammer2020} for Asplund spaces.
\bibliographystyle{siam}
\bibliography{references}

\end{document}